\documentclass{amsart}
\usepackage{amssymb,amscd,color}
\usepackage{tikz, tikz-cd}
\usepackage{xypic}
\usepackage{soul}

\usepackage{color}

\usepackage[
urlcolor=blue,
colorlinks=true,
linkcolor=blue,
citecolor=blue,
]{hyperref}

\newtheorem{thm}{Theorem}[section]
 \newtheorem{conj}[thm]{Conjecture}
\newtheorem{lem}[thm]{Lemma}

\newtheorem{cor}[thm]{Corollary}
\theoremstyle{definition}
\newtheorem{defn}[thm]{Definition}

\newtheorem{rmk}[thm]{Remark}

\newtheorem*{ack}{Acknowledgments}  
\theoremstyle{remark}

\numberwithin{equation}{section}

\newcommand{\mi}{\mathcal{I}}
\newcommand{\mo}{\mathcal{O}}

\newcommand{\C}{\mathbb{C}}

\newcommand{\Q}{\mathbb{Q}}

\newcommand{\exc}{\mathrm{Exc}}

\newcommand{\nklt}{\mathrm{Nklt}}
\newcommand{\nqklt}{\mathrm{Nqklt}}

\newcommand{\restr}[1]{{\raisebox{-0.0\height}{$\mid_{#1}$}}}

\begin{document}
\title[Simple connectedness of 
slc Fano log pairs]{Simple Connectedness of 
Fano log pairs with semi-log canonical singularities}
\author{Osamu Fujino}
\address{Osamu Fujino\\Department 
of Mathematics\\Graduate School of 
Science\\Osaka University\\Toyonaka\\Osaka 560-0043\\Japan}
\email{fujino@math.sci.osaka-u.ac.jp}
\author{Wenfei Liu}
\address{Wenfei Liu \\School of Mathematical 
Sciences\\ Xiamen University\\Siming 
South Road 422\\361005 Xiamen\\ Fujian\\ P. R. China}
\email{wliu@xmu.edu.cn}
\thanks{}
\subjclass[2010]{Primary 14J45; Secondary 14E30}
\date{2017/12/9}
\keywords{simple connectedness, 
rational chain connectedness, 
Fano varieties, semi-log canonical 
singularities}

\begin{abstract}
We show that any union of slc 
strata of a Fano log pair with semi-log canonical 
singularities is simply connected. 
In particular, Fano log pairs with semi-log canonical singularities are 
simply connected, which confirms a conjecture of the first author. 
\end{abstract}
\maketitle

\section*{Introduction}
Fano manifolds are complex projective manifolds 
whose canonical class is anti-ample. It is known 
that every Fano manifold is simply connected. Indeed, 
there are at least three independent 
proofs of this fact, we refer to \cite{Tak00} 
and \cite[Section 6]{Fuj14b} for more details. In 
birational geometry, the notion of Fano manifolds 
is generalized to that of Fano log pairs which 
allows singularities coming up naturally in the 
minimal model program (MMP). The rational 
chain connectedness of Fano log pairs with 
more and more general (up to log canonical) singularities was then 
established in \cite{Cam92, KMM92, Zha06, HM07}, 
and this implies that the (topological) fundamental 
group of the variety is finite. On the other hand, one shows that 
the \emph{algebraic} fundamental group of a 
Fano log pair is trivial by vanishing theorems. 
Combining these two facts, the simple 
connectedness of Fano log pairs with log 
canonical singularities follows (\cite[Theorem~6.1]{Fuj17b}). 

Semi-log canonical singularities incorporate the 
non-normal counterpart of log-canonical 
singularities. They appear on the varieties at 
the boundaries of the compactifications of 
moduli spaces (see \cite[Part III]{HK10}, 
\cite{Kol13a, Kol13b}). Even more general 
singularities, the quasi-log canonical singularities, 
are introduced in the inductive treatment of the 
MMP (\cite{Amb03}). This class of singularities allow 
to put log pairs with semi-log canonical singularities 
and their slc strata, or any union thereof, on the equal 
footing. The fundamental theorems 
in the MMP, especially the vanishing theorems, 
are now available in this context (\cite{Fuj17a}). 
As a consequence, quasi-log canonical varieties 
with anti-ample quasi-log canonical class has 
trivial algebraic fundamental group (\cite[Corollary 1.2]{Fuj17b}). This 
leads to the following

\begin{conj}[\cite{Fuj17b}, Conjecture 1.3]\label{conj}
Let $[X, \omega]$ be a projective quasi-log canonical 
pair such that $-\omega$ is ample. Then $X$ is simply connected.
\end{conj}

In this paper we confirm the conjecture 
for an important special case. It is an answer to \cite[Problem 3.6]{Fuj14b}. 

\begin{thm}[see Theorem \ref{thm2.7}]\label{thm: main}
Any union of slc strata of a Fano log 
pair $(X, \Delta)$ with semi-log canonical singularities is simply connected. 
\end{thm}

\noindent As a corollary, we solve \cite[Conjecture 1.4]{Fuj17b}:

\begin{cor}\label{cor0.3}
Fano log pairs with semi-log canonical singularities are simply connected.
\end{cor}

A key ingredient of the proof is a 
subadjunction formula for slc 
strata (Lemma~\ref{lem: subadj}). 
In particular, this makes the minimal slc 
stratum into a Fano log pair with 
Kawamata log terminal singularities, so it is simply connected. 

What also follows is that any union of 
slc strata of a Fano log pair with semi-log canonical 
singularities are rationally chain 
connected (Corollary~\ref{cor: rcc}). However, 
for non-normal varieties rational chain 
connectedness does not imply the finiteness 
of the fundamental group as in the 
normal case (consider for example 
a rational curve with nodes \cite[Remark 6.2]{Fuj17b}).

Thus we need to invoke \cite[Corollary 1.4]{HM07} and 
the van Kampen theorem to show that the natural homomorphism 
of fundamental groups induced by the 
inclusion of the minimal slc stratum 
into a union of slc strata is surjective. 
In this way the required simple connectedness is proved.

\medskip

\noindent{\bf Conventions:} We work 
over $\mathbb C$, the complex number field, throughout 
this paper. A \emph{scheme} means a separated scheme of finite type over 
$\mathbb C$. 
We freely use the standard notation of the MMP 
as in \cite{Fuj17a}. If $f\colon X\rightarrow Y$ is 
a continuous map between two path-connected topological spaces, we omit 
the base points for the fundamental groups in the induced 
homomorphism $\pi_1(X)\rightarrow\pi_1(Y)$, which 
will be harmless for the arguments in this paper. 
When we treat a Fano log pair $(X, \Delta)$ with 
semi-log canonical singularities, we always assume that 
$X$ is connected. 

\begin{ack}
The first author was partially 
supported by JSPS KAKENHI Grant 
Numbers JP16H03925, JP16H06337. The 
second author was partially supported by 
the NSFC (No.~11501012, No.~11771294) 
and by the Recruitment Program for Young Professionals.
The authors would like to thank Ms.~Kimiko Tanaka for 
her supports. 
\end{ack}

\section{Preliminary}
A log pair $(X,\Delta)$ consists of an 
equi-dimensional demi-normal 
scheme 
$X$ together with an effective $\mathbb R$-divisor 
$\Delta$ on $X$, such that $\Delta$ does 
not contain any irreducible components 
of the non-normal locus of $X$ and $K_X+\Delta$ is $\mathbb R$-Cartier. 
We note that 
a scheme $X$ is demi-normal if 
it satisfies Serre's $S_2$ condition and if its codimension one points
are either regular points or nodes (\cite[Denition 5.1]{Kol13b}).

\begin{defn}\label{def1.1}
A projective log pair $(X, \Delta)$ is 
Fano if $-(K_X+\Delta)$ is ample, or 
put another way, if $K_X+\Delta$ is anti-ample.
\end{defn}

Let $(X, \Delta)$ be a normal log pair. 
Let $f\colon Y\rightarrow X$ be a resolution 
such that $\exc(f)\cup f^{-1}_*\Delta$ has a 
simple normal crossing support, where $\exc(f)$ is 
the exceptional locus of $f$ and $f^{-1}_*\Delta$ is 
the strict transform of $\Delta$ on Y. We can write
\[
K_Y = f^*(K_X+\Delta) +\sum_i a_iE_i.
\]
We usually write $a_i=a(E_i, X,\Delta)$ and 
call it the \emph{discrepancy} of $E_i$ with 
respect to $(X,\Delta)$. We say that $(X, \Delta)$ 
is \emph{log canonical} (resp.~\emph{Kawamata 
log terminal}) if $a_i\geq -1$ (resp.~$a_i>-1$) for every $i$. 
We use abbreviations \emph{lc} and $\emph{klt}$ for 
log canonical and Kawamata log terminal respectively.

If $(X,\Delta)$ is a normal log pair 
(resp.~an lc pair)
and if there exist a resolution $f\colon Y\rightarrow X$ 
and a prime divisor $E$ on $Y$ such 
that $a(E, X,\Delta)\leq -1$ (resp.~$a(E, X,\Delta)= -1$) 
then $f(E)$ is called a \emph{$\nklt$ center} 
(resp.~\emph{an lc center}) of $(X,\Delta)$. 
The \emph{$\nklt$ locus} of $(X,\Delta)$, 
denoted by $\nklt(X, \Delta)$, is the union of 
all $\nklt$ centers. An \emph{lc stratum} of an lc pair
$(X,\Delta)$ means either an lc center or an irreducible component of $X$. 

\begin{defn}\label{def1.2}
A log pair $(X,\Delta)$ is said to have \emph{semi-log 
canonical} (\emph{slc}) singularities if  $(\bar X, \Delta_{\bar X})$ has
 log canonical singularities, where 
 $\nu\colon\bar X\rightarrow X$ is the 
 normalization and 
 $K_{\bar X}+\Delta_{\bar X}=\nu^*(K_X+\Delta)$. 
 An \emph{slc center} of $(X,\Delta)$ is the 
 image of an lc center of $(\bar X, \Delta_{\bar X})$. 
 An \emph{slc stratum} of $(X, \Delta)$ means 
 either an slc center of $(X,\Delta)$ or an irreducible 
 component of $X$. \end{defn}

Note that an slc stratum is irreducible by definition.

\medskip  

Let $D=\sum _i d_i D_i$ be an $\mathbb R$-divisor, where 
$D_i$ is a prime divisor and $d_i\in \mathbb R$ for every $i$ 
such that $D_i\ne D_j$ for $i\ne j$. 
We put 
$$
D^{<1} =\sum _{d_i<1}d_iD_i, \quad 
D^{\leq 1}=\sum _{d_i\leq 1} d_i D_i, \quad 
D^{= 1}=\sum _{d_i= 1} D_i, \quad 
\text{and} \quad \lceil D\rceil =\sum _i \lceil d_i \rceil D_i, 
$$ 
where $\lceil d_i\rceil$ is the integer defined by $d_i\leq 
\lceil d_i\rceil <d_i+1$. 

\medskip

Let $Z$ be a simple normal crossing 
divisor on a smooth variety $M$ and $B$ an $\mathbb R$-divisor 
on $M$ such that 
$Z$ and $B$ have no common irreducible components and 
that the support of $Z+B$ is a simple normal crossing divisor on $M$. In this 
situation, $(Z, B|_Z)$ is called a {\em{globally embedded simple 
normal crossing pair}}.

\medskip

Let us quickly look at the definition of 
qlc pairs for the reader's convenience.

\begin{defn}[Qlc pairs]\label{def1.3}
Let $X$ be a scheme and $\omega$ an 
$\mathbb R$-Cartier divisor (or an $\mathbb R$-line bundle) on $X$. 
Let $f:Z\to X$ be a proper morphism from a globally embedded 
simple normal 
crossing pair $(Z, \Delta_Z)$. If the natural map 
$\mathcal O_X\to f_*\mathcal O_Z(\lceil -(\Delta_Z^{<1})\rceil)$ is an 
isomorphism and $f^*\omega\sim _{\mathbb R} K_Z+\Delta_Z$, 
then $[X, \omega]$ is called a {\em{quasi-log canonical pair}} 
({\em{qlc pair}}, for short). 
\end{defn}

\begin{rmk}\label{rem1.4}
We can define {\em{qlc centers}}, 
{\em{qlc strata}}, $\nqklt(X, \omega)$, and so on,  
for a qlc pair $[X, \omega]$, 
which are 
counterparts of 
(s)lc centers, (s)lc strata, and 
$\nklt(X, \Delta)$, respectively. 
In the situation of Definition \ref{def1.3}, 
$C$ is a qlc stratum of $[X, \omega]$ if and only if $C$ is the 
$f$-image of some slc stratum of $(Z, \Delta^{=1}_Z)$. 
The subvariety $C$ is a qlc center of $[X, \omega]$ if and only if 
$C$ is a 
qlc stratum of $[X, \omega]$ but is not an irreducible component of 
$X$. 
The union of all qlc centers of $[X, \omega]$ is denoted by 
$\nqklt(X, \omega)$. 
For the details, see \cite[Definitions 
6.2.2, 6.2.8, 6.2.9, and 
Notation 6.3.10]{Fuj17a}. 
\end{rmk}

We refer to \cite[Chapter 6]{Fuj17a} for 
the theory of quasi-log schemes. 

\begin{thm}[{see \cite[Theorem 1.2]{Fuj14a}}]\label{thm1.5}
Let $(X, \Delta)$ be a quasi-projective 
log pair with 
semi-log canonical singularities. 
Then $[X, \omega]$, where $\omega=K_X+\Delta$, 
has a qlc structure which is compatible with 
the original slc structure of $(X, \Delta)$. 
This means that 
$C$ is an slc center {\em{(}}resp.~slc stratum{\em{)}} 
of $(X, \Delta)$ if 
and only if $C$ is a qlc center {\em{(}}resp.~qlc stratum{\em{)}} 
of $[X, \omega]$. 
In particular, any union of slc strata of $(X, \Delta)$ is qlc 
by adjuction. 
\end{thm} 
For the details of Theorem \ref{thm1.5}, 
see \cite{Fuj14a}. 
By this theorem, we can apply the theory of quasi-log schemes in 
\cite[Chapter 6]{Fuj17a} to log pairs with semi-log canonical singularities. 

\medskip

Let $(X,\Delta)$ be a quasi-projective 
log pair with semi-log canonical singularities. Then 
its slc strata have some nice properties (\cite[Theorem~6.3.11]{Fuj17a}):
\begin{enumerate}
\item[(a)] there is a unique minimal slc stratum through a given point;
\item[(b)] the minimal slc stratum 
at a given point is normal at that point;
\item[(c)] the intersection of two slc strata is a union of slc strata.
\end{enumerate} 
If $(X,\Delta)$ is additionally Fano, that is, $X$ is projective and 
$-(K_X+\Delta)$ is 
ample, then 
\begin{itemize}
\item[(d)] any union of slc strata of $(X,\Delta)$ is connected;
\item[(e)] there is a unique minimal slc 
stratum of $(X,\Delta)$,  which is normal. 
\end{itemize}
For (d) it suffices to show that 
$H^0(W, \mo_W)=\C$ for any union 
$W$ of slc strata, which follows from 
the vanishing $H^1(X, \mi_W)=0$ (\cite[Theorem~1.11]{Fuj14a}), 
and (e) is a direct consequence of (a), (b) and (d).

\section{Proof}
For the proof of Theorem \ref{thm: main}, we may assume that 
$\Delta$ is a $\mathbb Q$-divisor by perturbing 
$\Delta$. Therefore, for simplicity, we assume that every divisor is a 
$\mathbb Q$-divisor from now on. Let us start with 
the following easy lemma. 

\begin{lem}\label{lem2.1}
Let $f:X\to Y$ be a surjective morphism between connected
normal projective varieties. 
Let $\Delta$ be an effective 
$\mathbb Q$-divisor on $X$ such that 
$(X, \Delta)$ is lc and is klt over the generic point of $Y$. 
Assume that $K_X+\Delta\sim _{\mathbb Q} f^*D$ for 
some $\mathbb Q$-Cartier divisor $D$ on $Y$. 
Then we can construct an effective $\mathbb Q$-divisor 
$\Delta_Y$ on $Y$ such that $ K_Y+\Delta_Y \sim_\Q D$ and 
$\nklt(Y,\Delta_Y)\subset f(\nklt (X, \Delta))$.  
\end{lem}

\begin{proof}
Let 
$$
f: X\overset{g}{\longrightarrow} Z \overset{h}{\longrightarrow} Y
$$ 
be the Stein factorization of $f:X\to Y$. 
By the theory of lc-trivial fibrations (see 
\cite[Theorem 0.2]{Amb04} and 
\cite[Theorem 3.3]{Amb05}), 
there exist a proper birational morphism 
$\sigma:Z'\to Z$ from a smooth 
projective 
variety $Z'$, a $\mathbb Q$-divisor $B_{Z'}$ on $Z'$, 
and a nef $\mathbb Q$-divisor 
$M_{Z'}$ on $Z'$ with the following properties:  
\begin{itemize}
\item[(i)] $\sigma^*h^*D\sim _{\mathbb Q} K_{Z'}+B_{Z'}+M_{Z'}$, 
\item[(ii)] the support of $B_{Z'}$ is a simple 
normal crossing divisor on $Z'$, $B_{Z'}=B^{\leq 1}_{Z'}$, 
and $B_{Z'}=B^{<1}_{Z'}$ outside 
$\sigma^{-1}(g(\nklt (X, \Delta)))$, and 
\item[(iii)] there exist a proper surjective 
morphism $p:Z'\to Z''$ onto a normal projective 
variety $Z''$ and a nef and big $\mathbb Q$-divisor 
$M_{Z''}$ on $Z''$ such that $M_{Z'}\sim_{\mathbb Q} p^*M_{Z''}$. 
\end{itemize}
Then we can take an effective $\mathbb Q$-divisor $G_{Z'}$ such that 
$G_{Z'}\sim _{\mathbb Q} M_{Z'}$ and 
that $K_Z+\Delta_Z$ is klt outside 
$g(\nklt (X, \Delta))$, where 
$\Delta_Z=\sigma_*(B_{Z'}+G_{Z'})$. 
We note that $K_Z+\Delta_Z\sim _{\mathbb Q} h^*D$ 
by construction. 
Then the proof of \cite[Lemma 1]{FG12}, applied to 
$h: Z\to Y$, gives an effective $\mathbb Q$-divisor $\Delta_Y$ on $Y$ 
such that $K_Y+\Delta_Y\sim _{\mathbb Q} D$ and 
that $\nklt(Y, \Delta_Y)\subset h(\nklt (Z, \Delta_Z))
\subset f(\nklt (X, \Delta))$.  
\end{proof}

\begin{rmk}\label{rem2.2}
In Lemma \ref{lem2.1}, 
it is easy to construct an effective $\mathbb Q$-divisor 
$\Delta_Y$ on $Y$ such that 
$K_Y+\Delta_Y\sim _{\mathbb Q} D+A$ with 
$\nklt (Y, \Delta_Y)\subset 
f(\nklt (X, \Delta_X))$, where $A$ is any ample $\mathbb Q$-divisor 
on $Y$ (see also \cite[Theorem 1.2]{Fuj99}). 
We can prove the above weaker statement 
without using Ambro's deep result (\cite[Theorem 3.3]{Amb05}). 
The nefness of $M_{Z'}$ (see \cite[Theorem 0.2]{Amb04}), which is 
much simpler than \cite[Theorem 3.3]{Amb05}, is sufficient. 
For the proof of Theorem \ref{thm: main}, 
we can replace $\Delta$ with 
$\Delta+\varepsilon H$, 
where $H$ is a general very ample effective 
divisor on $X$ and $\varepsilon$ is a sufficiently 
small positive 
rational number. 
Therefore, we can prove Theorem \ref{thm: main} 
without using \cite[Theorem 3.3]{Amb05}. 
\end{rmk}

As an application of Lemma \ref{lem2.1}, we have the following lemma. 

\begin{lem}\label{lem: subadj}
Let $W$ be an slc stratum of a projective log 
pair $(X,\Delta)$ with semi-log canonical singularities, and let $E$ be the 
union of all slc strata that are strictly contained in $W$.
Let $\nu\colon \bar W\rightarrow W$ 
be the normalization. Then there is an 
effective $\Q$-divisor $B_{\bar W}$ 
on $\bar W$ such that 
$K_{\bar W}+B_{\bar W}\sim_\Q (K_X+\Delta)
\restr{\bar W}$ and $\nklt(\bar W, B_{\bar W})\subset \nu^{-1}(E)$. 
Moreover, if $(X, \Delta)$ is additionally Fano, then 
$\nklt(\bar W, B_{\bar W})$ is connected.
\end{lem}

\begin{proof}
Let $\mu\colon \bar X\rightarrow X$ 
be the normalization of $X$. Let 
$\bar X_i$ be an irreducible component 
of $\bar X$ that contains an irreducible 
component $V$ of $\mu^{-1}(W)$. Let 
$\Delta_{\bar X_i}$ be the effective $\Q$-divisor 
defined by $K_{\bar X_i}+\Delta_{\bar X_i}= 
(K_X+\Delta_X)\restr{{\bar X_i}}$. 
Then $(\bar X_i,\Delta_{\bar X_i})$ has log canonical singularities.

  Let $f\colon (Y, \Delta_Y)\rightarrow (\bar X_i, \Delta_{\bar X_i})$ 
 be a $\Q$-factorial dlt blow-up such that $K_Y+\Delta_Y
 =f^*(K_{\bar X_i}+\Delta_{\bar X_i})$ 
(see, for example, \cite[Theorem 4.4.21]{Fuj17a}). There 
 is an lc stratum $S$ of $(Y,\Delta_Y)$ 
 dominating $V$. We take $S$ to be a minimal 
 such lc stratum. Then $S$ is normal, and if $\Delta_S$ 
 is the effective $\Q$-divisor on $S$ defined 
 by adjunction $K_S+\Delta_S=(K_Y+\Delta_Y)\restr{S}$ 
 then $(S,\Delta_S)$ is again a dlt pair. 

Let $\sigma\colon\bar V\rightarrow V$ be 
the normalization. Then the 
morphism $\mu\restr{V}\circ \sigma\colon \bar V\rightarrow W$ 
factors through a 
morphism $\bar\mu\colon \bar V\rightarrow \bar W$. 
Since $S$ is normal, the morphism 
$S\rightarrow V$ factors through a 
morphism $g\colon S\rightarrow \bar V$. 
Thus we have a commutative 
diagram
$$
\xymatrix{
S\ar[d]_g\ar[rd]&\\
\bar V \ar[r]^{\sigma}\ar[d]_{\bar\mu} & V \ar[d]^{\mu\restr{V}}\\
\bar W \ar[r]^{\nu} & W
}
$$
where the morphisms in the lower square are all finite.

By the choice of $S$, the log 
pair $(S,\Delta_S)$ has klt singularities 
over the generic point of $\bar W$. By 
applying Lemma \ref{lem2.1} to $\bar\mu\circ g: S\to \bar W$, 
we can take an effective 
$\mathbb Q$-divisor $B_{\bar W}$ 
on $\bar W$ such that $K_{\bar W}+B_{\bar W}\sim 
_{\mathbb Q} (K_X+\Delta)|_{\bar W}$ and that the following inclusions 
$$
\nklt(\bar W, B_{\bar W})\subset \bar \mu\circ g(\nklt (S, \Delta_S))\subset 
\nu^{-1}(E)
$$ hold. 

If $-(K_X+\Delta)$ is ample, 
then so is $-(K_{\bar W}+B_{\bar W})$. 
Therefore, by the Nadel vanishing theorem (see, for example, 
\cite[Theorem 3.4.2]{Fuj17a}), we obtain 
$$H^i(\bar W, \mathcal J(\bar W, B_{\bar W}))=0$$ 
for any $i>0$, where $\mathcal J(\bar W, B_{\bar W})$ is the multiplier 
ideal sheaf of $(\bar W, B_{\bar W})$. 
It follows from the long 
exact sequence of cohomology 
that the natural restriction map 
$$H^0(\bar W, \mo_{\bar W})\rightarrow 
H^0(\nklt(\bar W, B_{\bar W}), 
\mo_{\nklt(\bar W, B_{\bar W})})$$ 
is surjective. Therefore, we see that  
$$H^0(\nklt(\bar W, B_{\bar W}), 
\mo_{\nklt(\bar W, B_{\bar W})})\cong\C$$ 
and $\nklt(\bar W, B_{\bar W})$ is connected.
\end{proof}

\begin{cor}\label{cor: min}
Let $(X,\Delta)$ be a Fano log pair with semi-log canonical singularities 
and $W_0$ its minimal slc stratum. Then there is 
an effective $\Q$-divisor $B_{W_0}$ on $W_0$ 
such that $(W_0, B_{W_0})$ is a Fano log 
pair with klt singularities and that 
\[
K_{W_0}+B_{W_0}\sim_\Q (K_X+\Delta)\restr{W_0}
\]
Hence $W_0$ is rationally connected {\em{(}}\cite{Zha06}{\em{)}} and 
$\pi_1(W_0)=1$ {\em{(}}\cite[Theorem 1.1]{Tak00}{\em{)}}.
\end{cor}
\begin{cor}\label{cor: rcc}
Let $(X,\Delta)$ be a Fano log pair with 
slc singularities. Then any union of slc 
strata of $(X,\Delta)$ is rationally chain connected.
\end{cor}
\begin{proof}
Since any union of slc strata of $(X,\Delta)$ is 
connected, it suffices to prove that 
any single slc stratum are rationally chain connected. 

Let $W$ be an slc stratum of $(X,\Delta)$ and $E$ 
the union of all slc strata that are strictly 
contained in $W$. Let $\nu\colon \bar W\rightarrow W$ 
be the normalization. By Lemma~\ref{lem: subadj} there is an 
effective $\Q$-divisor $B_{\bar W}$ on $\bar W$ such 
that $K_{\bar W}+B_{\bar W}\sim_\Q (K_X+\Delta)\restr{\bar W}$, 
which is anti-ample, and $\nklt(\bar W, 
B_{\bar W})\subset \nu^{-1}(E)$. 
By \cite[Corollary 1.4]{BP11}, $\bar W$ is rationally 
connected modulo $\nklt (\bar W, B_{\bar W})$, 
that is, for any general point $w$ of $\bar W$ there 
exists a rational curve $C_w$ passing through 
$w$ and intersecting $\nklt(\bar W, B_{\bar W}$).  
In particular, $\bar W$ is rationally connected modulo $\nu^{-1}(E)$. 
It follows that $W$ is 
rationally connected 
modulo $E$ which is the union of 
lower dimensional slc strata. Thus 
we can run induction on dimension of the slc strata, 
noting that the minimal slc stratum of $(X,\Delta)$ 
is rationally connected by Corollary~\ref{cor: min}.
\end{proof}

We prepare an important lemma for the proof of Theorem 
\ref{thm2.7}. 

\begin{lem}\label{lem: conn}
Let $W$ be a non-minimal slc stratum 
of a Fano log pair $(X,\Delta)$ with 
semi-log canonical singularities, and 
let $E$ be the union of all slc strata that 
are strictly contained in $W$.
Let $\nu\colon \bar W\rightarrow W$ 
be the normalization. Then $\nu^{-1}(E)$ is connected.
\end{lem}

\begin{proof}
By \cite[Theorem 1.2]{Fuj14a} (see also Theorem 
\ref{thm1.5}) and adjunction (see, 
for example, \cite[Theorem 6.3.5 (i)]{Fuj17a}), 
$[W, \omega]$ is a qlc pair, where $\omega=(K_X+\Delta)|_W$, 
such that $\nqklt(W, \omega)=E$. 
By \cite[Theorem 1.1]{FL17}, we see that 
$[\bar W, \nu^*\omega]$ is also qlc with 
$\nqklt(\bar W, \nu^*\omega)=\nu^{-1}(E)$. 
Since $-\nu^*\omega$ is ample, 
$H^i(\bar W, \mathcal I_{\nqklt(\bar W, \nu^*\omega)})=0$ for 
every $i>0$ by the vanishing theorem 
(see, for example, \cite[Theorem 6.3.5 (ii)]{Fuj17a}). 
Note that $\mathcal I_{\nqklt (\bar W, \nu^*\omega)}$ is the defining 
ideal sheaf of $\nqklt (\bar W, \nu^*\omega)$ on $\bar W$. 
It follows from the long 
exact sequence of cohomology 
that the natural restriction map $$H^0(\bar W, \mo_{\bar W})\rightarrow 
H^0(\nqklt(\bar W, \nu^*\omega), 
\mo_{\nqklt(\bar W, \nu^*\omega)})$$ 
is surjective. Therefore, we see that 
$
H^0(\nqklt(\bar W, \nu^*\omega), 
\mo_{\nqklt(\bar W, \nu^*\omega)})
\cong\C$ 
and $\nu^{-1}(E)=\nqklt(\bar W, \nu^*\omega)$ is connected.
\end{proof}

\begin{thm}\label{thm2.7}
Let $(X,\Delta)$ be a Fano log pair with slc singularities 
and $W$ the union of some slc strata of $(X,\Delta)$.  Then $\pi_1(W)=1$.
\end{thm}
\begin{proof}
Note that $W$ is connected and contains 
the minimal slc stratum $W_0$ of $(X,\Delta)$.
Let $W^{(0)}:=W$. Suppose that $W^{(i)}$ is 
defined. We define $W^{(i+1)}$ to be the 
union of slc strata that are strictly contained 
in an irreducible component of $W^{(i)}$. 
Thus we obtain a filtration of reduced subschemes of $W$:
\[
W=W^{(0)}\supset W^{(1)}\supset\cdots\supset W^{(k)}=W_0
\]
 We want to show by inverse induction 
 on $i$ that $\pi_1(W^{(i)})=1$ for 
 any $i \geq 0$. In particular, it will follow that $\pi_1(W)=\pi_1(W^{(0)})=1$.
  
 By Corollary~\ref{cor: min} we know that $\pi_1(W^{(k)})=\pi_1(W_0)=1$.

 Now assuming $\pi_1(W^{(i)})=1$, we 
 need to show that 
 $\pi_1(W^{(i-1)})=1$. For 
 simplicity of notation, let us 
 denote $Z:=W^{(i-1)}$and 
 $E:=W^{(i)}$. We construct a 
 covering family of open 
 subsets of $Z$ in the 
 Euclidean topology: Let $Z_j$ be 
 the irreducible components of $Z$ 
 and $U_j$ an open neighborhood 
 of $E_j:=Z_j \cap E$ in $Z_j $ such that $E_j$ is 
 a deformation retract of $U_j$ 
 (cf.~\cite[Chapter~1, 
Theorem~8.8]{BHPV04}). Let $U=\cup_j U_j$ 
and $V_j:=Z_j \cup U$. Then $\{V_j\}_j$ is an 
open covering of $Z$ such that $V_{j_1}\cap V_{j_2}=U$ 
for any $j_1\neq j_2$, which is connected. 

Note that the $U_j$'s are closed subsets of $U$ 
and $U_{j_1}\cap U_{j_2}=E_{j_1}\cap E_{j_2}$ 
for $j_1\neq j_2$, so the deformation 
retractions $U_{j}\times I \rightarrow U_{j}$ 
from $U_{j}$ onto $E_j$ coincide on $(U_{j_1}\cap U_{j_2})\times I$ 
for any $j_1\neq j_2$, and thus glue to a continuous 
map $U\times I\rightarrow U$ which is a deformation retraction of $U$ 
onto $E$. Here $I$ denotes the unit interval $[0,1]$. 
Similarly, there is a deformation 
retraction from $V_j$ onto $Z_j\cup E$. It 
follows that $\pi_1(U)=\pi_1(E)=1$ and $\pi_1(V_j)=\pi_1(Z_j\cup E)$. Applying 
the van Kampen theorem, we obtain an isomorphism 
\begin{equation}\label{eq: fund 1}
*_j \pi_1(V_j)\xrightarrow{\sim}\pi_1(Z).
\end{equation}
where $*_j \pi_1(V_j)$ denotes the 
free product of the $\pi_1(V_j)$'s. 
Therefore, it suffices to show 
that $\pi_1(V_j)=\pi_1(Z_j\cup E)$ is trivial for each $j$.

Let $\nu\colon \bar Z_j \rightarrow Z_j $ be the 
normalization. Then by Lemma~\ref{lem: subadj} we 
can find an effective $\mathbb Q$-divisor $B_{\bar Z_j }$ on 
$\bar Z_j $ such that 
$K_{\bar Z_j }+B_{\bar Z_j }\sim_\Q (K_X+\Delta)\restr{\bar Z_j }$, 
which is anti-ample. We note that 
$\nu^{-1}(E_j)$ is connected  by Lemma \ref{lem: conn}. 
By \cite[Proposition~3.1]{FPR15}, there is a homotopy equivalence 
between the double mapping cylinder 
$E\,\cup_\nu \nu^{-1}(E_j)\times I\, \cup_\iota\bar Z_j$ 
and $Z_j\cup E$ where $\iota\colon \nu^{-1}(E_j)\hookrightarrow \bar Z_j$ 
is the inclusion map, and it follows that (cf.~\cite[Corollary 3.2, (ii)]{FPR15})
\begin{equation}\label{eq: FPR}
\pi_1(Z_j\cup E)\cong \pi_1(E\,\cup_\nu \nu^{-1}(E_j)
\times I\, \cup_\iota\bar Z_j)\cong \pi_1(E) *_{\pi_1(\nu^{-1}(E_j))}\pi_1(\bar Z_j).
\end{equation}
By \cite[Corollary 1.4]{HM07} the induced homomorphism 
$\pi_1(\nklt(\bar Z_j, B_{\bar Z_j}))\rightarrow \pi_1(\bar Z_j)$
is surjective. Therefore, $\pi_1(\nu^{-1}(E_j))\rightarrow \pi_1(\bar Z_j)$ 
is surjective since 
$\nklt(\bar Z_j, B_{\bar Z_j})\subset \nu^{-1}(E_j)\subset 
\bar Z_j$, 
and so is the induced 
homomorphism 
$\pi_1(E)\rightarrow \pi_1(Z_j\cup E)$ 
by \eqref{eq: FPR}.  Since $\pi_1(E)=1$ by the induction hypothesis, 
we have the desired triviality of $\pi_1(Z_j\cup E)$.
\end{proof}

\end{document}